\documentclass{aptpub}

\usepackage{latexsym}
\usepackage{amssymb}
\usepackage{amsfonts}
\usepackage{amsmath}
\usepackage{eucal}

\authornames{A.D.~Barbour, P.~Chigansky and F.C.~Klebaner} 
\shorttitle{Random initial conditions in fluid limits} 

\def\uK{^{(K)}}
\def\l{\lambda}
\def\m{\mu}
\def\a{\alpha}
\def\h{\eta}
\def\Eq{\ =\ }
\def\Def{\ :=\ }
\def\eq{\begin{equation}}
\def\en{\end{equation}}
\def\eqa{\begin{eqnarray}}
\def\ena{\end{eqnarray}}
\def\eqs{\begin{eqnarray*}}
\def\ens{\end{eqnarray*}}
\def\oZ{{\overbar{Z}}}

\def\oU{{\overbar{U}}}
\def\oV{{\overbar{V}}}
\def\tg{{\tilde g}}

\def\Le{\ \le\ }
\def\hM{{\widehat M}}
\def\d{\delta}
\def\t{\tau}
\def\e{\varepsilon}
\def\a{\alpha}
\def\b{\beta}
\def\g{\gamma}
\def\half{\tfrac12}
\def\non{\nonumber}

\def\th{\theta}
\def\Ref#1{(\ref{#1})}
\def\uab{^{(\g,\b)}}
\def\uabM{^{(\g,\b;M)}}

\def\Bl{\left(}
\def\Br{\right)}
\def\Blb{\left\{}
\def\Brb{\right\}}
\def\tphi{{\tilde\phi}}

\def\hZ{{\widehat Z}}

\renewcommand{\P}{\mathbb{P}}
\renewcommand{\E}{\mathbb{E}}
\newcommand{\one}[1]{\mathbf{1}_{\{#1\}}}
\newcommand{\Real}{\mathbb R}
\newcommand{\overbar}[1]{\mkern 1.5mu\overline{\mkern-1.5mu#1\mkern-1.5mu}\mkern 1.5mu}

\numberwithin{equation}{section}  

\begin{document}

\title{On the emergence of random initial conditions in fluid limits} 

\authorone[Universit\"at Z\"urich]{A. D. Barbour} 

\addressone{Institut f\"ur Mathematik, Universit\"at Z\"urich,
Winterthurertrasse 190, CH-8057 Z\"URICH; Email address: a.d.barbour@math.uzh.ch} 

\authortwo[The Hebrew University of Jerusalem]{P. Chigansky} 

\addresstwo{Department of Statistics,
The Hebrew University,
Mount Scopus, Jerusalem 91905,
Israel; Email address: pchiga@mscc.huji.ac.il} 

\authorthree[Monash University]{F. C. Klebaner} 

\addressthree{School of Mathematical Sciences,
Monash University, Monash, VIC 3800, Australia; Email address: fima.klebaner@monash.edu} 

\begin{abstract}
The paper presents a phenomenon occurring in population
processes that start near zero and have large carrying capacity.
By the classical result of Kurtz~(1970), such processes, normalized by the carrying capacity, converge on finite
intervals to the solutions of ordinary differential equations, also known as the fluid limit.
When the initial population is small relative to carrying capacity, this limit is trivial.
Here we show that, viewed at suitably chosen times increasing to infinity, the process converges to the fluid limit,
governed by the same dynamics, but with a random initial condition.
This random initial condition is related to the martingale limit of an associated linear
birth and death process.
\end{abstract}

\keywords{birth-death process, population dynamics with carrying capacity, fluid approximation} 


\section{Introduction}

Many models of population growth can be formulated, following the ideas of McKendrick~\cite{McK26} and
Bartlett~\cite{Bartlett49}, \cite{Bartlett60}, as Markovian birth and death (BD) processes.  The classical
Malthusian model can be viewed as a BD process with
population birth rate $\lambda z$ and death rate $\mu z$ depending linearly on the population size~$z$,
corresponding to
constant {\it per capita\/} birth rate~$\l$ and death rate~$\m$.  This process
cannot stabilize near any finite population size, and so non-linear density dependent BD processes~$(Z_t,\,t\ge 0)$,
with {\it per capita\/} birth
rates $\lambda  - (\l-\m)g_1(z/K)$ and death rates $\mu   + (\lambda-\mu)g_2(z/K)$, $z\in \mathbb{Z}_+$,
have been introduced to remedy the defect.  In such a formulation, $\lambda>\mu\ge 0$ are fixed constants,
$g = g_1 + g_2$ is typically an increasing
function with $g(0)=0$ and $g(x_\infty)=1$ for some $x_\infty\in (0,\infty)$,
and $K$ is a parameter, thought of as being large, that is representative of the {\em carrying capacity}.

The analogue of Verhulst's~(1838) model has $g_1(x) = 0$ and $g_2(x) = x$ for all $x \ge 0$, and is
known as the stochastic logistic process;
it serves as our prototype. Ricker's~\cite{Ri54} model has $g_1(x) = \frac{\l}{\l-\m}(1 - e^{-\a x})$ and $g_2(x) = 0$;
that of Beverton \& Holt~\cite{BH57} has $g_1(x) = \frac{\l}{\l-\m} x/(x+m)$ and $g_2(x) = 0$, that of
Hassell~\cite{Hass75} has
$g_1(x) = \frac{\l}{\l-\m}\big\{1 - (1+x/m)^{-c}\big\}$ and $g_2(x) = 0$, and that of Maynard--Smith \&
Slatkin~\cite{MSS73} has
$g_1(x) = \frac{\l}{\l-\m}\big\{ 1 - (1 + (x/m)^c)^{-1}\big\}$ and $g_2(x) = 0$.

In these models, when~$K$ is large and the initial population size~$Z_0$ is  relatively small, the birth rate
exceeds the death rate, and the population size begins by
growing exponentially,  avoiding extinction in the early stages with a significant probability. As the size gets
larger, the net birth rate decreases and
population growth slows down, settling around the carrying capacity $K x_\infty$. The population
typically fluctuates around the carrying capacity for a very long period of time, until, by chance,
it eventually dies out.

This qualitative behavior can be made precise by considering the normalized {\em density} process $\overbar{Z}_t= Z_t/K$.
By the result of  Kurtz~\cite{Kurtz}, Theorem 2.11, for any fixed $T>0$,
\begin{equation}\label{lim}
    \sup_{0\le t\le T}\left|\overbar{Z}_t -x_t\right|\xrightarrow[K\to\infty]{d}0,
\end{equation}
where $x=(x_t)_{t\in \Real_+}$ is the solution of the o.d.e., or fluid limit,
\begin{equation}\label{ode}
    \dot x_t \Eq  (\lambda -\mu)x_t\big(1 -g(x_t)\big), \quad t\ge 0,
\end{equation}
subject to the initial condition $x_0:=\lim_{K\to\infty}\overbar{Z}_0$.

When the initial population size~$Z_0$ is proportional to~$K$, the initial condition $x_0$ is positive,
and the density process
$\overbar{Z}$ converges to the corresponding positive solution of~\eqref{ode}. In particular, this implies that
extinction  prior to any {\em fixed} time $T$ has vanishing probability.
As $T$ increases, the solution of \eqref{ode} approaches its stable equilibrium at $x_\infty$.
Since $\overbar{Z}$ is a transient Markov chain, it is absorbed at zero eventually. However, large deviation analysis
(see, for example,  Barbour~\cite{b76} and  Jagers \& Klebaner~\cite{JagKle2}) shows that
$\overbar{Z}$ does not leave a vicinity
of~$x_\infty$ for a long period of time, with mean growing exponentially with~$K$.

If the initial population size $Z_0$ is fixed with respect to $K$, so that $x_0=0$, the limit~\eqref{lim} implies
that~$\overbar{Z}$ converges to the zero function on any bounded interval. This implies that those
trajectories of~$Z$ that stay positive up to time~$T$ remain of smaller order than~$K$ during that time, so that
it takes longer to grow to level comparable with~$K$. Other than that, the convergence~\eqref{lim}
reveals no information about the behaviour of those trajectories that eventually reach the carrying capacity.

In the present paper, we derive a limit theorem showing that, if the initial population is small when
compared to~$K$, so that $x_0=0$, the density process nonetheless converges over increasing time intervals to a
nontrivial solution of the same o.d.e.~\eqref{ode}, but now with a {\em random} initial condition.

The emergence of a random initial condition in the limit can already be seen in the simple model of pure birth processes.
This case admits a one page proof, involving nothing more complicated than weighted sums of
i.i.d.\ exponential random variables (Section \ref{sec-bp}).
A completely different approach is required in the more general setup of Theorem~\ref{thm}.
Here, the proof relies on the approximation of the non-linear BD process by a linear BD process during the initial
stages, and by the non-linear deterministic dynamics thereafter.

As pointed out in  Barbour {\it et al.\/}~\cite{BHKK15}, the idea of such an approximation is not new,
going back to the papers of
Kendall~\cite{Kendall56} and Whittle~\cite{Whittle55} in the mid 1950's. However, its rigorous justification in
many of the models where it has heuristically been invoked
can be quite involved. Non-linear multidimensional Markov population processes were considered recently in \cite{BHKK15},
where it was established that, after an initial build up phase, the random population follows the solution
of the corresponding deterministic equations, but with a {\em random} time shift (\cite{BHKK15}, Theorem~1.1).
The proof in~\cite{BHKK15} relies on an abstract coupling construction (Thorisson~\cite{Thor}, Theorem 7.3).

 Here, we revisit the one-dimensional setting, in which the argument can be made much simpler;
in particular, there is a very neat explicit expression for the random initial condition
to be used with the fluid approximation. In addition, the argument can be carried through
under somewhat weaker assumptions than are used in~\cite{BHKK15}.

\section{The main result}

Defining $g_l^+(x) := \sup_{0\le y\le x}|g_l(y)|$, $l=1,2$,
and recalling that $g = g_1 + g_2$, we work under the following assumptions:
\eqa
     (\mathrm{i})&& g(0) = 0,\ g(x_\infty) = 1 \mbox{ for some}\ x_\infty < \infty,
              \mbox{ and } g(x) < 1 \mbox{ for } 0 < x < x_\infty; \non\\
    (\mathrm{ii})&&  xg(x) \mbox{ is uniformly Lipschitz on } [0,x_\infty], \mbox{ with constant }\th \ge 1;
                       \label{ADB-assns} \\
    (\mathrm{iii})&& g^+(x) := g_1^+(x) + g_2^+(x) \mbox{ is such that } x^{-1}g^+(x) \mbox{ is integrable
                 from } 0. \non
\ena
In view of \Ref{ADB-assns}~(ii), the o.d.e.~\eqref{ode}, with initial condition $x_s = x$, has a unique solution.
It is given implicitly by
\eq\label{ADB-ode-solution}
   G(x_t) - G(x) \Def \int_{x}^{x_t} \frac{du}{u(1-g(u))} \Eq (\l - \m)(t-s),
\en
where the function~$G$ is determined up to an additive constant;
for any~$0 < a < x_\infty$, we can for instance take
\eq \label{ADB-Gdef}
   G(x) \Eq G_a(x) \Def \int_a^x \frac{du}{u(1-g(u))} + \log a \Eq \log x + H_a(x) ,
\en
with
$$
   H_a(x) \Def \int_{a}^{x} \frac{g(u)\,du}{u(1-g(u))}.
$$
With this notation, we can formulate
our main result as follows.

\begin{thm}\label{thm}
For $\l > \m > 0$ and for $0 \le \a < 1$,
let $(Z\uK,\,K \ge 1)$ be a sequence of BD processes with {\it per capita\/} birth rates $\lambda  - (\lambda-\mu) g_1(z/K)$
and death rates  $\mu  + (\lambda-\mu) g_2(z/K)$,  started at the initial population size
$Z_0=\lfloor K^\alpha \rfloor$.
Let $\oZ\uK(t) := K^{-1}Z\uK(t)$ and $t_1(K) := (\l-\m)^{-1}\log{K^{1-\a}}$.
Then,  under Assumptions~\Ref{ADB-assns}, the sequence of processes
$\oZ\uK(t_1(K)+\cdot)$  converges weakly as $K\to\infty$, in the uniform topology
on bounded intervals, to the solution of the
o.d.e.~\eqref{ode} started with the initial condition
\begin{equation}\label{RIC}
  w_0 \Def \begin{cases}
        G_0^{-1}\big(\log W\big), &  \alpha = 0; \\
        G_0^{-1}\big(0\big), & \alpha \in (0,1),
     \end{cases}
\end{equation}
where $W$ is a random variable with
$$
\P[W = 0] = \m/\l\quad \text{and}\quad \P[W > w] = (1 - \m/\l)e^{-(1 - \m/\l)w}, \quad w\ge 0.
$$
\end{thm}

\medskip

\begin{rem}\label{ADB-remark}
\
\medskip

{\bf (1)}    The function $G_0$ is well defined, because of~Assumption~\Ref{ADB-assns}~(iii), and
is strictly increasing, having $\lim_{x\to 0+}G_0(x)=-\infty$ and
$\lim_{x\to x_\infty} G_0(x)=\infty$. The latter limit holds, because
$0 \le x_\infty - xg(x) \le \th(x_\infty - x)$ in $0 \le x \le x_\infty$, from Assumption~\Ref{ADB-assns}~(ii),
implying that
$$
0 \Le s - sg(s) \Eq (s - sg(s)) - (x_\infty - x_\infty g(x_\infty)) \Le (\th-1) (x_\infty - s), \quad 0 < s < x_\infty,
$$
and that $sg(s) \ge \half x_\infty$ in $(1-1/2\th)x_\infty \le s \le x_\infty$.
From this it follows that
\eqs
    \int_{(1-1/2\th)x_\infty}^x \frac{g(s)}{s(1-g(s))}\,ds
 &>& \frac1{x_\infty}\int_{(1-1/2\th)x_\infty}^x \frac{sg(s)}{s-sg(s)}\,ds
      \\
 &>& \   \frac12 \int_{(1-1/2\th)x_\infty}^x \frac{ds}{(\th-1)(x_\infty-s)} \\
 &=& \frac1{2(\th-1)} \log\Bigl(\frac{x_\infty/2\th}{x_\infty - x}\Bigr)\
               \xrightarrow[x\to x_\infty]{}\ \infty.
\ens
Hence~$G := G_0$ is a bijection from $(0,x_\infty)$ to $\Real$, with bounded continuous inverse
$G^{-1}\colon\Real\mapsto (0,x_\infty)$.

In particular, $g(x)=x^p$ with $p>0$, satisfies our assumptions, with $G(x)=\frac 1 p \log \frac{x^p}{1-x^p}$, giving
$$
  w_0 \Eq \begin{cases}
      \left(\frac{W^p}{1+W^p}\right)^{\frac 1 p}, &  \alpha = 0; \\
      \left(\frac 1 2\right)^{\frac 1 p}, & \alpha \in (0,1).
  \end{cases}
$$
The stochastic logistic process corresponds to taking $p=1$, and yields the initial condition
$w_0=\frac{W}{1+W}$ in \eqref{ode}.

\medskip
{\bf (2)}
It follows from \Ref{ADB-BD} below that~$W$
has the distribution of the a.s.~limit of the  martingale $e^{-(\lambda-\mu)t}Y_t$, when~$Y$ is the linear BD
process with {\it per capita\/} birth and death rates $\lambda$ and $\mu$, starting with $Y_0 = 1 = K^0$.
If~$Y\uK$ denotes the same process, but with initial condition $Y\uK_0 = \lfloor K^\a \rfloor$
for some $0 < \a < 1$, then the martingale $e^{-(\lambda-\mu)t}(Y\uK_t/Y\uK_0)$ has mean~$1$ and variance of
order~$K^{-\a}$ as $K\to\infty$, explaining why~$W$ is replaced by~$1$ in~\eqref{RIC} when $\a \in (0,1)$.

\medskip
{\bf (3)}
Theorem \ref{thm} implies that the trajectories that survive early extinction reach the magnitude of the
carrying capacity at times of order $\frac 1 {\lambda-\mu}\log K^{1-\alpha}$.  For $\alpha=0$,
since $G^{-1}(-\infty)=0$, it follows
from~\eqref{RIC} that the trajectories that vanish are those corresponding to the set $\{W=0\}$.
This set is exactly the set of extinction  of the linear branching process~$Y$. For $\alpha>0$,
the probability of early extinction vanishes as $K\to\infty$ for both $Z\uK$ and~$Y\uK$.

\medskip
{\bf (4)}
The Lipschitz assumption on the function~$xg(x)$ can be
replaced by assuming that it is increasing, and has finite derivative at~$x_\infty$:
see Remark~\ref{ADB-increasing}.

\end{rem}

\section{A preview: pure Birth Process}\label{sec-bp}

This subsection is a short detour from our main setup, which provides an additional insight into the structure of the limit.
Consider a non-linear pure birth process $Z$ that jumps from an integer $z$ to $z+1$ at rate $\lambda(z)=\lambda z(1-g(z/K))$, $z=1,2,\ldots, [K x_\infty]$,
where $\lambda>0$ is a constant and $g$ is a function satisfying the assumption of Theorem \ref{thm}. Let $Z_0=1$ and
define  $\lambda(z)=0$ for $z> [Kx_\infty]$, so that $Z$ is absorbed, once it exceeds the level $[Kx_\infty]$.
The holding time in state $z$ equals $\tau_z/\lambda(z)$, where $\tau_z\sim \mathrm{Exp}(1)$ and $\tau_z$'s are i.i.d.
Consider also a linear pure birth  process $Y$ with $Y_0=1$ and birth rates $\lambda z$, $z\in \mathbb{Z}_+$. It is well known that $e^{-\lambda t}Y_t$ is an
$L^2$  bounded martingale which has an almost sure limit $W$\nocite{Fimabook}, and that $W$ has $\mathrm{Exp}(1)$ distribution.
\begin{prop}
Let $G$ be defined as in Theorem \ref{thm},  and   $Z_0=1$ then
\begin{equation}\label{limBP}
\frac{1}{K}Z_{\frac{1}{\lambda}\log K}\xrightarrow[K\to\infty]{d} G^{-1}(\log W).
\end{equation}
\end{prop}

\begin{proof}
Let $Y$ and $Z$ be defined as above, using the same sequence of random variables $(\tau_i)$.
Due to monotonicity of a pure birth process for $t\ge 0$
\begin{align*}
&
\big\{Z_t> n\big\}=\left\{\sum_{i=1}^{n}\frac{1}{\lambda(i)}\tau_i < t\right\}=\big\{T_n < t\big\}, \quad n \le  [x_\infty K] \\
&
\big\{Y_t > n\big\}=\left\{\sum_{i=1}^{n}\frac{1}{\lambda \,i}\tau_i < t\right\}=\big\{\widetilde T_n <  t\big\},\quad n\in \mathbb{N},
\end{align*}
where $T_n$ and $\widetilde T_n$ are the times of the $n$-th jump of $Z$ and $Y$ respectively:
\begin{equation}\label{sums}
T_n=\sum_{i=1}^{n}\frac{1}{\lambda(i)}\tau_i\quad \text{ and}\quad  \widetilde T_n=\sum_{i=1}^{n}\frac{1}{\lambda \, i}\tau_i.
\end{equation}
Note that the coefficients in the first sum $T_n$ in \eqref{sums} depend on $K$, whereas in the second sum $\widetilde T_n$ they do not. Therefore we establish convergence of the second sum first, and then show that their difference converges to a constant.
Since $Y_{\widetilde T_n}=n$, $\lim_{n\to\infty}\widetilde T_n=\infty$ and $\lim_{t\to\infty} e^{-\lambda t}Y_t =W$
\begin{equation}\label{limTilde}
\widetilde T_n -\frac 1 \lambda \log n = -\frac 1 \lambda \log \left(e^{-\lambda\widetilde T_n} Y_{\widetilde T_n}\right)\xrightarrow[n\to\infty]{a.s.}-\frac 1 \lambda\log W.
\end{equation}
Let us show that for any $x\in (0,x_\infty)$
\begin{equation}\label{Delta}
T_{[xK]} -\widetilde T_{[xK]} \xrightarrow[K\to\infty]{L^2}  \frac 1 \lambda \int_0^x \frac{1}{s}\frac {g(s)}{1-g(s)}ds.
\end{equation}
Indeed, denoting by  $h(s)= \frac {g(s)}{s(1-g(s))}$, we have
\begin{align*}
\lambda\,\E \left(T_{[xK]} -\widetilde T_{[xK]}\right)   =
&  \sum_{i=1}^{[xK]} \frac 1{i} \frac{g(i/K)}{1-g(i/K)}\E  \tau_i
=\sum_{i=1}^{[xK]} \frac 1{i/K} \frac{g(i/K)}{1-g(i/K)} \frac 1 K\\
& =\sum_{i=1}^{[xK]} \frac 1 K h(i/K)
\xrightarrow[K\to\infty]{} \int_0^x h(s)ds,
\end{align*}
where we used  Assumption~\Ref{ADB-assns}~(iii). Similarly,
\begin{equation}
\label{var}
\begin{aligned}
\lambda^2\var \left(T_{[xK]} -\widetilde T_{[xK]}\right)
&
=\sum_{i=1}^{[xK]} \left(\frac 1{i} \frac{g(i/K)}{1-g(i/K)}\right)^2
=\frac{1}{K}\sum_{i=1}^{[xK]}\frac{1}{K} h^2(i/K)\xrightarrow[K\to\infty]{}0.
 \end{aligned}
\end{equation}
This can be seen as follows. Let $\varepsilon>0$ be given. Choose $\delta>0$ such that $0\le g(x)\le \varepsilon$ for all $x\in[0,\delta]$. This is possible because $g$ is continuous and $g(0)=0$. It is clear that $\frac{1}{K}\sum_{i=[\delta K]}^{[xK]}\frac{1}{K} h^2(i/K)\xrightarrow[K\to\infty]{}0$, because the function $h^2(s)=(g(s)/s)^2$ is bounded and integrable on $[\delta,x]$.
The residual sum satisfies
$$
\frac{1}{K}\sum_{i=1}^{[\delta K]}\frac{1}{K} h^2(i/K) \le \frac{1}{K}\Big(\max_{1\le i\le [\delta K]}h(i/K)\Big)\sum_{i=1}^{[\delta K]}\frac{1}{K}  h(i/K).
$$
By Assumption~\Ref{ADB-assns}~(iii), the sum in the right hand side converges to  $\int_0^\delta h(s)ds <\infty$
and
$$
\frac{1}{K}\Big(\max_{1\le i\le [\delta K]}h(i/K)\Big)\le C \frac{1}{K} \max_{1\le i\le [\delta K]}\frac{g (i/K)}{i/K}\le C\max_{1\le i\le [\delta K]}g (i/K)\le C\varepsilon,
$$
with a constant $C$ independent of $K$. Thus the convergence in \eqref{var} holds by arbitrariness of $\varepsilon$ and the limit \eqref{Delta}
follows.

Now \eqref{limTilde} and  \eqref{Delta} imply
$$
T_{[xK]}-\frac{1}{\lambda }\log K \xrightarrow[K\to\infty]{\P}
\frac 1 \lambda \left(\int_0^x \frac{1}{s}\frac {g(s)}{1-g(s)}ds +   \log x- \log W  \right) = \frac 1\lambda\Big(G(x)- \log W\Big).
$$
Since $W$ has a continuous distribution,
\begin{align*}
\P\left(\frac{1}{K}Z_{\frac{1}{\lambda}\log K}>x\right) =\;
&
\P\left( Z_{\frac{1}{\lambda}\log K}> [xK]\right) = \P\left(T_{[xK]} < \frac{1}{\lambda}\log K\right)
\xrightarrow[K\to\infty]{}\\
&
 \P\Big(G(x)- \log W <0\Big) =\P\Big(G^{-1}(\log W)>x\Big), \quad \forall x\in (0,x_\infty)
\end{align*}
which proves  \eqref{limBP}.
\end{proof}

\section{Proof of Theorem \ref{thm}}

The main idea of the proof is to construct the process  $Z\uK$, together with an auxiliary linear BD
process~$Y\uK$, on the same probability space, in such a way that~$Z\uK$ is well approximated by~$Y\uK$ on the interval
$[0,t_0(K)]$, where $t_0(K) := \frac 1 {\lambda-\mu}\log K^c$, and $c>0$ is a constant
such that $\a + c$ is less than, but close enough to~$1$: more precisely, such that
\eq\label{ADB-c-def}
   0 < \{1-(\a+c)\}(1+\th) < 1/2,
\en
for~$\th$ as in Assumption~\Ref{ADB-assns}~(ii).

Thereafter, we extrapolate this approximation on $[t_0(K),t_1(K)]$, using the flow generated by the o.d.e.~\eqref{ode}.
Our proof shows that this approximation is enough to establish Theorem~\ref{thm}.
The main effort is in proving that
\begin{equation}\label{teqz}
    \frac 1 K Z_{ t_1}\ \xrightarrow[K\to \infty]{d}\ w_0,
\end{equation}
where~$w_0$ is as in~\eqref{RIC}.
Once this is done, the rest is immediate from Kurtz~\cite{Kurtz}, Theorem~2.11.

\medskip

To this end,
for each~$K$, we  construct a process $(Y\uK,Z\uK,U\uK,V\uK)$ with the following properties:

\medskip

\begin{enumerate}
\addtolength{\itemsep}{0.7\baselineskip}
\renewcommand{\theenumi}{\alph{enumi}}

\item  $Y\uK_0=Z\uK_0=U\uK_0=V\uK_0= \lfloor K^\alpha \rfloor$;

\item  $Y\uK$ is the linear BD process  with {\it per capita\/}  birth rate $\lambda$ and death rate $\mu$;

\item  $Z\uK$ is the non-linear BD process with {\it per capita\/} birth rate $\lambda - (\lambda-\mu) g_1(z/K)$
     and  death rate $\mu + (\lambda-\mu) g_2(z/K)$, $z\in \mathbb{Z}_+$;

\item  $U\uK$ is the linear BD process  with {\it per capita\/} birth rate $\{\lambda + \l_K\}$
  and death rate $\{\mu  -  \m_K\}$, where
\[
     \l_K \Def (\lambda-\mu) g_1^+\big(K^{\a+c+\h-1}\big),\quad \m_K \Def (\lambda-\mu) g_2^+\big(K^{\a+c+\h-1}\big),
\]
and  where~$\h$ is a constant satisfying $0 < \h < 1-\a-c$;

\item  $V\uK$ is the linear BD process  with {\it per capita\/} birth rate $\{\lambda - \l_K\}$
  and death rate $\{\mu  +  \m_K\}$;

\item  $V\uK_t \le Y\uK_t \le U\uK_t$ for all $t\ge 0$;

\item  $V\uK_t\le Z\uK_t \le U\uK_t$ for $t\le \tau\uK$, where $\tau\uK$ is the first time at which~$Z\uK$
hits the level  $K^{\a + c+ \h}$:
\begin{equation}\label{tau}
  \tau\uK \Def \inf\left\{t\ge 0: Z\uK_t \ge K^{\a + c+ \h}\right\}.
\end{equation}

\end{enumerate}

\noindent
The coupling is described in Section~\ref{coupling}.

\medskip

Suppressing the dependence on~$K$ where possible,
define $\overbar{Y}_t := \frac 1 K Y_t$ and another auxiliary process
\begin{equation}\label{tildeZ}
   \widetilde Z_t \Def
              \begin{cases}
                     \overbar{Y}_t, & t\le t_0; \\
                     \phi_{t_0,t}(\overbar{Y}_{t_0}), & t > t_0,
              \end{cases}
\end{equation}
where $\phi_{s,t}(x)$ is the flow generated by the o.d.e.~\eqref{ode};  that is, using~\eqref{ADB-ode-solution},
\eq\label{ADB-detflow2}
   G(\phi_{s,t}(x)) - G(x) \Eq (\lambda-\mu)(t-s),
\en
if $x > 0$, and $\phi_{s,t}(0) = 0$ for all $t > s$.

It thus follows from \eqref{tildeZ}, with our choices of $t_0$ and $t_1$, that,
on the set $\{\overbar{Y}_{t_0} >0\}$,
$$ 
   G(\widetilde Z_{t_1}) \Eq G(\overbar{Y}_{t_0})   +  (1-\alpha-c)\log K;
$$
the equation is also trivially true when $\{\overbar{Y}_{t_0} = 0\}$, since then
$\widetilde Z_{t_1} = \overbar Y_{t_0} = 0$, and $G(0) = -\infty$.
Thus, introducing $W_K := K^{1-\alpha-c} \chi_\a(K)\overbar{Y}_{t_0} = e^{-(\l-\m)t_0}(Y_{t_0}/Y_0)$,
with $\chi_\a(K):=K^\a/\lfloor K^\a \rfloor$, we obtain
\begin{equation}\label{ADB-GW1}
               G(\widetilde Z_{t_1})  \Eq  \log  W_K - \log \chi_\a(K) + H\bigl(K^{\alpha+c-1}W_K/\chi_\a(K)\bigr).
\end{equation}

It follows from Remark~\ref{ADB-remark}\,(2) that
\begin{equation}\label{Kc}
  W_K \Eq e^{-(\l-\m)t_0(K)}(Y\uK_{t_0(K)}/Y\uK_0)\ \xrightarrow[K\to\infty]{d}\
     \begin{cases}
           W, & \text{if } \alpha =0; \\
           1 & \text{if } \alpha \in (0,1).
     \end{cases}
\end{equation}
Hence, since~$H$ is continuous and $H(0)=0$, since $\alpha+c<1$ and since $\lim_{K\to\infty}\chi_\a(K) = 1$,
the last term in~\eqref{ADB-GW1}
converges in distribution to zero as $K\to \infty$. Thus, again using~\eqref{Kc} in~\eqref{ADB-GW1},
and because the function~$G^{-1}$ is continuous, it follows that
$$ 
     \widetilde Z\uK_{t_1(K)}\xrightarrow[K\to\infty]{d} G^{-1}\big(\one{\alpha=0}\log W\big).
$$

It remains to show that $\widetilde Z\uK$ is an appropriate approximation for $\overbar{Z}\uK$ at time $t_1(K)$;
we use the coupling to show that
$$ 
    \overbar{Z}\uK_{t_1(K)} - \widetilde Z\uK_{t_1(K)}\ \xrightarrow[K\to\infty]{d}\ 0.
$$
Since
$$
   \bigl|\overbar{Z}_{t_1}-\widetilde Z_{t_1}\bigr| \Eq
                   \left|\oZ_{t_1} -\phi_{t_0,t_1}(\overbar{Y}_{t_0})\right|
     \Le \left|\oZ_{t_1} - \phi_{t_0,t_1}(\overbar{Z}_{t_0})\right| +
                    \left|\phi_{t_0,t_1}(\overbar{Z}_{t_0}) - \phi_{t_0,t_1}(\overbar{Y}_{t_0})\right|,
$$
it is enough to show that
\begin{equation}\label{raz}
            \oZ\uK_{t_1(K)} - \phi_{t_0,t_1}(\overbar{Z}\uK_{t_0})\ \xrightarrow[K\to\infty]{d}\ 0,
\end{equation}
and, using the coupling of $Y\uK,Z\uK,U\uK$ and~$V\uK$ constructed in Section~\ref{coupling}, that
\begin{equation}\label{dva}
     \phi_{t_0,t_1}(\overbar{Z}\uK_{t_0}) - \phi_{t_0,t_1}(\overbar{Y}\uK_{t_0})\ \xrightarrow[K\to\infty]{d}\ 0.
\end{equation}
These two relations are proved in Sections \ref{proofraz} and~\ref{proofdva}.  The proof of~\eqref{teqz}
is then complete.

Before proving \eqref{raz} and~\eqref{dva}, we collect some useful facts.
First, for $0 < x < 1$, we have
\[
    \half g^+(x) \log(1/x) \Le \int_{x}^{\sqrt x} \frac{g^+(s)}s\,ds \Le \int_0^{\sqrt x} \frac{g^+(s)}s\,ds,
\]
so that, from Assumption~\Ref{ADB-assns}~(iii), $\lim_{x\to 0} g^+(x) \log(1/x) = 0$.
This, in particular, implies that $\lim_{K\to\infty} g_l^+(K^{-\g})\log K = 0$ for any $\g > 0$
and $l \in \{1,2\}$,
and hence that
\eq\label{ADB-K-limits}
  \lim_{K\to\infty} (\l_K + \m_K)\log K \Eq 0;
\en
it also follows that~$g$ is continuous at~$0$.
Then, in view of Assumptions~\Ref{ADB-assns}~(i) and~(ii), we have
\eq\label{ADB-f-def}
    xg(x) \ \ge\ \max\{-\th x, x_\infty + \th(x-x_\infty)\} \ =:\ f_g(x), \quad 0 \le x \le x_\infty.
\en
Note that~$f_g$ is convex, and that $\tg(x) := x^{-1}f_g(x)$ is increasing, since $\th \ge 1$.

Next, let $Z\uab := (Z_t\uab,\,t\ge0)$ denote the linear birth and death process with
{\it per capita\/} birth rate~$\g$ and death rate~$\b$, and with $Z_0\uab = 1$; suppose
that $\g > \b$.  Then, writing $\h_t := (\b/\g)e^{-(\g-\b)t}$, we have
\eqa
   \frac\b\g - \P[Z_t\uab = 0] &=& \frac{\h_t}{1 - \h_t}\Bl 1 - \frac\b\g \Br;\non\\
    \P[Z_t\uab > r] &=& \frac{1-\b/\g}{1-\h_t}\Blb 1 - \Bl \frac\g\b - 1\Br\,\frac{\h_t}{1 - \h_t} \Brb^r,
          \quad r \ge 1; \label{ADB-BD}
\ena
furthermore, $\E Z_t\uab = e^{(\g-\b)t}$ and $\var Z_t\uab \le \frac{\g+\b}{\g-\b}e^{2(\g-\b)t}$
 (see p.\,159 in \cite{GS01}).
The process $(Z_t\uabM,\,t\ge0)$ with the same birth and death rates, but starting with $Z_0\uabM = M$,
is distributed as the sum of~$M$ independent copies of~$Z\uab$; hence, by Chebyshev's inequality,
\eq\label{ADB-BD2}
   \P[|M^{-1}e^{-(\g-\b)t} Z_t\uabM - 1| \ge \e] \Le \frac{\g+\b}{M\e^2(\g-\b)}.
\en

Now recall the well known semimartingale decomposition of $Z$:
\begin{equation}\label{Zsem}
          Z_t \Eq Z_0 + (\lambda -\mu)\int_0^t \Big(Z_s -   Z_s g(Z_s/K)\Big)ds  + M_t, \quad t\ge 0,
\end{equation}
(see, for example, Klebaner~\cite{Fimabook} p.360), where $M=(M_t)_{t\ge 0}$ is a martingale with $M_0=0$ a.s.\ and
$$
         \langle M\rangle_t  \Eq \int_0^t \Big((\lambda +\mu)Z_s
                   +  (\lambda -\mu) Z_s  (g_2(Z_s/K) - g_1(Z_s/K)) \Big)\,ds.
$$
Dividing both sides of~\eqref{Zsem} by~$K$, we see that the density  process $\oZ_t= K^{-1} Z_t$ satisfies the equation
\begin{equation}\label{sde}
 \overbar{Z}_t \Eq \overbar{Z}_0 + (\lambda -\mu)\int_0^t \Big(\overbar{Z}_s -\overbar{Z}_s g(\overbar{Z}_s)\Big)\,ds  +
             \frac 1 {\sqrt{K}} \hM_t, \quad t\ge 0,
\end{equation}
where the martingale $\hM$ has zero mean and predictable quadratic variation
\eqa\label{barM}
    \langle\hM\rangle_t
    &=&  \int_0^t \Big((\lambda +\mu)\oZ_s +  (\lambda -\mu)\oZ_s (g_2(\oZ_s) - g_1(\oZ_s)\Big)\,ds  \non\\
     &=&  \int_0^t \Big((\lambda +\mu)\overbar{Z}_s +  (\lambda -\mu)\overbar{Z}_sg(\overbar{Z}_s)\Big)\,ds.
\ena
Taking expectations in~\eqref{sde}, and recalling~\Ref{ADB-f-def},
we see that
\begin{align*}
 \E\overbar{Z}_t & =\,  \E\overbar{Z}_0 + (\lambda -\mu)\int_0^t \Big(\E\overbar{Z}_s
            -\E\{\overbar{Z}_s g(\overbar{Z}_s)\}\Big)\,ds   \\
   & \le\, \E\overbar{Z}_0 + (\lambda -\mu)\int_0^t \Big(\E\overbar{Z}_s
            -\E f_g(\overbar{Z}_s)\Big)\,ds   \\
  &\le\,  \E\overbar{Z}_0 + (\lambda -\mu)\int_0^t \Big(\E\overbar{Z}_s
            -f_g(\E \overbar{Z}_s)\Big)\,ds ,
\end{align*}
this last because~$f_g$ is convex.
Hence $\E\overbar{Z}_t$ satisfies the integral inequality
$$
     \E\overbar{Z}_t \Le \E\oZ_0 + (\l-\m)\int_0^t \Bigl(\E\overbar{Z}_s -\E\overbar{Z}_s \tg(\E\overbar{Z}_s)\Big)ds,
$$
so that $\E\overbar{Z}_t \le \tphi_{0,t}(\E\oZ_0)$, where $\tphi_{0,t}(x)$ is the flow
generated by replacing $g$ by~$\tg$ in the o.d.e.~\Ref{ode}.
Thus,
in particular, since $\tg(x_\infty)=1$, $0\le \E\overbar{Z}_t \le  x_\infty$ for
all $t\ge 0$.  This in turn implies, using~\eqref{sde}, that
\eq\label{ADB-ZgZ-mean}
      (\lambda -\mu)\int_{t_0}^t \E\{\overbar{Z}_s g(\overbar{Z}_s)\}\,ds \Eq
          \E\overbar{Z}_{t_0}+ (\lambda -\mu)\int_{t_0}^t \E\overbar{Z}_s\,ds - \E\overbar{Z}_t
           \Le x_\infty\{1 + (\lambda -\mu)(t-t_0)\} .
\en

\subsection{Proof of \eqref{raz}}\label{proofraz}
Write $\delta_t := \overbar{Z}_{t}-\widehat Z_t$, where $\widehat Z_t := \phi_{t_0,t}(\overbar{Z}_{t_0})$  satisfies the equation
$$
   \widehat Z_t \Eq
       \overbar{Z}_{t_0} + (\lambda -\mu)\int_{t_0}^{t} \Big(\widehat{Z}_s -\widehat{Z}_s g(\widehat{Z}_s)\Big)ds,
            \quad t\ge t_0,
$$
so that,  using \eqref{sde},
$$
  \delta_t \Eq  (\lambda -\mu)\int_{t_0}^{t} \left( \delta_s
    + \widehat{Z}_s g( \widehat{Z}_s) -\overbar{Z}_s g(\overbar{Z}_s)  \right) ds
      + \frac 1 {\sqrt{K}} \big(\hM_{t} - \hM_{t_0}\big).
$$
Applying It\^o's formula to $\d_t^2$ as a function of~$\d_t$, (see, e.g.,   eq. (8.58) \cite{Fimabook})  we obtain
\eq\label{ADB-Ito}
 \delta_t^2 \Eq
   2(\lambda -\mu)\int_{t_0}^{t} \left(   \delta^2_s
      - \big(\overbar{Z}_s g(\oZ_s) -\widehat{Z}_s g( \widehat{Z}_s)\big)(\oZ_s -\widehat{Z}_s)  \right) ds
          + \frac 1 K \sum_{t_0\le s\le t}  \left(\Delta \hM_s\right)^2.
\en

Taking expectations of both sides, and using Assumption~\Ref{ADB-assns} (ii),
we obtain the inequality
\begin{equation}\label{dineq}
  \E \delta^2_t \Le 2(\lambda -\mu)(1+\th)\int_{t_0}^{t} \E \delta^2_s\, ds
     + \frac 1 K \int_{t_0}^t \E\Big((\lambda +\mu) \overbar{Z}_s +  (\lambda -\mu) \overbar{Z}_sg(\overbar{Z}_s)\Big)ds,
           \quad t\ge t_0;
\end{equation}
for the last integral,  we have used
the formulae $\sum_{s\le t} (\Delta \hM_s)^2=\bigl[\hM\bigr]_t $ and
$\E \bigl[\hM \bigr]_t=\E \bigl\langle \hM \bigr\rangle_t$, together with~\eqref{barM}.
Using $\E \oZ_s \le x_\infty$ and~\eqref{ADB-ZgZ-mean} in~\eqref{dineq}, we obtain
$$ 
      \E \delta_t^2 \Le 2(\lambda -\mu)(1+\th)\int_{t_0}^{t} \E \delta^2_s\, ds
              + \frac 1 K x_\infty\{1 + 2\l(t-t_0)\}.
$$
The Gr\"{o}nwall inequality now yields
$$
        \E \delta_{t_1}^2 \Le \frac 1 K x_\infty\{1 + 2\l(t_1-t_0)\} e^{2(\lambda -\mu)(1+\th)(t_1-t_0)}.
$$
Since $(\lambda-\mu)(t_1-t_0)= \log K^{1-\alpha-c}$ and by the choice~\Ref{ADB-c-def} of~$c$,
it follows that
$$
      \E \delta_{t_1}^2 \Le  x_\infty\left\{1 + \frac{2\l}{\l-\m}\log K\right\} K^{2(1+\th)(1-\alpha-c)-1}
           \ \xrightarrow[K\to\infty]{}\ 0,
$$
and \eqref{raz} is proved.

\subsection{Proof of \eqref{dva}}\label{proofdva}
In this section, we use the coupling of $(Y,Z,U,V)$ established in Section~\ref{coupling}.

First, we show that $\lim_{K\to\infty}\P[\t\uK \le t_0(K)] = 0$, where $\t\uK$ is as in~\eqref{tau}.
Because, from property~(g), $Z_t \le U_t$ for all $0 \le t\le t\uK$, we have
\begin{align*}
 \P[\t\uK \le t_0(K)] &\Le
           \P\Bigl(\sup_{0 \le t\le t_0} U_t \ge K^{\a + c + \h}\Bigr)  \\
  & \Le \P\Bigl(\sup_{0 \le t\le t_0} e^{-\g_K t}U_t \ge e^{-\g_K t_0} K^{\a + c + \h}\Bigr), 
\end{align*}
where $\g_K := \l - \m + \l_K + \m_K$ is the exponential growth rate of the birth and death process~$U$.
Applying Doob's inequality to the martingale $e^{-\g_K t}U_t$ thus shows that
\[
    \P[\t\uK \le t_0(K)] \Le K^\a \,K^{-(\a + c + \h)}e^{\g_K t_0} \ \sim\ K^{-\h} \ \to 0,
\]
as $K \to \infty$, because $(\l_K + \m_K)\log K \to 0$ from~\Ref{ADB-K-limits}.
In view of~\Ref{ADB-detflow2} and of properties (f) and~(g) of the coupling, it is thus
enough for~\eqref{dva} to show that
\eq\label{ADB-flows}
     \phi_{t_0,t_1}(\oU\uK_{t_0}) - \phi_{t_0,t_1}(\oV\uK_{t_0})\ \xrightarrow[K\to\infty]{d}\ 0,
\en
where $\oU\uK_t := K^{-1}U\uK_t$ and $\oV\uK_t := K^{-1}V\uK_t$.

If $\a = 0$, by \eqref{ADB-ode-solution} and \eqref{ADB-Gdef}, and on the set $\{U_{t_0}>0\}$,
we have
$$
G\big(\phi_{t_0,t_1}(\oU_{t_0})\big) \Eq \log \oU_{t_0} + H(\oU_{t_0})   + (\l - \m)(t_1-t_0)
  \Eq \log (K^{ -c} U_{t_0}) + H(\oU_{t_0}).
$$
Define
$$
\Psi(x) \ :=\
\begin{cases}
  G^{-1}(\log x), & x>0; \\
  0, & x=0.
\end{cases}
$$
This is a continuous function from $[0,\infty)$ to $[0,x_\infty)$, and
\eq\label{PCH-flows}
\phi_{t_0,t_1}(\oU_{t_0}) \Eq \Psi\left(K^{-c}U_{t_0}e^{H(\oU_{t_0})}\right);
  \quad \phi_{t_0,t_1}(\oV_{t_0}) \Eq \Psi\left(K^{-c}V_{t_0}e^{H(\oU_{t_0})}\right),
\en
irrespective of whether $U_{t_0}$ and~$V_{t_0}$ are zero or positive.
Now, from~\Ref{ADB-BD} and~\Ref{ADB-K-limits},
\begin{equation}\label{PCH-U-limits}
\begin{aligned}
&
\lim_{K\to\infty}\P[U\uK_{t_0(K)} = 0] \Eq \frac\m\l  \\
&
\lim_{K\to\infty}\P[K^{-c}U\uK_{t_0(K)} > x] \Eq (1-\m/\l)\exp\{-x(1-\m/\l)\}, \quad x > 0.
\end{aligned}
\end{equation}
The same argument then shows that the limits~\Ref{PCH-U-limits} also hold if $U_{t_0}$
is replaced by~$V_{t_0}$.
Since $K^{-c}U_{t_0} \ge K^{-c}V_{t_0}$ a.s., and they both have the
same limits in distribution, it follows that $K^{-c}(U_{t_0}-V_{t_0}) \stackrel{d}{\to} 0$ as $K\to\infty$.
Note also, from~\Ref{PCH-U-limits}, that $\oU_{t_0} \stackrel{d}{\to} 0$
 as $K\to\infty$, and thus $e^{H(\oU_{t_0})} \stackrel{d}{\to} 1$ also; and that the same relations are
true if $\oU$ is replaced by~$\oV$.
From~\Ref{PCH-flows} and from the continuity of~$\Psi$, the convergence~\Ref{ADB-flows} now follows.
For $0 < \a < 1$, $K^{-c}$ is replaced by $K^{-c-\a}$ in~\Ref{PCH-flows}, and the convergence in distribution of
both $K^{-c-\a}U_{t_0}$ and $K^{-c-\a}V_{t_0}$ to the constant~$1$ follows from~\Ref{ADB-BD2},
in view of~\Ref{ADB-K-limits}.

\begin{rem}\label{ADB-increasing}
Assumption~\Ref{ADB-assns}~(ii) is used to justify \Ref{dineq} on the basis of~\Ref{ADB-Ito}, to
guarantee the uniqueness of the solutions of~\Ref{ode}, and to show that~$G_0$ maps to the whole
of~$\Real$.  However, if $xg(x)$ is non-decreasing in~$x$, it follows from~\Ref{ADB-Ito} that~\Ref{dineq}
holds with~$\th$ replaced by zero.  Furthermore, if $\oZ$ is replaced by any solution of the o.d.e.~\Ref{ode}
other than~$\hZ$, but also starting at~$\oZ_{t_0}$, the difference $\d_t := \oZ_t - \hZ_t$ satisfies~\Ref{ADB-Ito},
with~$M$ the zero function, from which it follows, using Gronwall's inequality, that $\d_t = 0$ for all
$t \ge t_0$, implying uniqueness of the solutions to the o.d.e.
\end{rem}

\subsection{Coupling birth and death processes} \label{coupling}

The proof of \eqref{teqz} will be completed, once we construct the processes $Y\uK$, $Z\uK$, $U\uK$ and $V\uK$, all
on the same probability space, with the properties (a)--(g).
The basic element of our construction is a coupling of two birth and death processes,
one of which has greater birth rate and smaller death rate than the other.
Such a coupling has been suggested, e.g., in~\cite{BBG86}.
For the sake of completeness we give a construction in much the same spirit.
As usual, we suppress the index~$K$ as far as we can.

The coupling is based on a collection of four  sequences of
independent standard Poisson processes~$(\Pi^l_n,\,n\ge1)$, $l \in \{1,2,3,4\}$, together with two double
sequences $(J_n^l(i),\,i\ge0, n\ge1)$, $l\in\{3,4\}$, of independent uniform $U[0,1]$ random variables,
all of which are mutually independent.  We then define processes $(J_{nt}^l,\,t\ge0)$ by
\[
     J_{nt}^3 \Def J_n^3(\Pi_n^3(2\l_K t));\quad J_{nt}^4 \Def J_n^4(\Pi_n^4(2\m_K t)),
\]
where $\l_K$ and~$\m_K$ are as defined in property~(d).  The definitions of the BD processes $U,Y$ and~$V$
are now simple to write down:
\begin{align*}
U_t   =  U_0 + \sum_{n\ge 1} \int_0^t & \one{n\le U_{s-}}\{d\Pi^1_n((\l - \l_K)s) + d\Pi_n^3(2\l_K s)\}  \\
     &
     -\sum_{n\ge 1} \int_0^t \one{n\le U_{s-}}d\Pi^2_n((\mu - \m_K)s);
\\
Y_t = Y_0 + \sum_{n\ge 1} \int_0^t & \one{n\le Y_{s-}}\{d\Pi^1_n((\l - \l_K)s) + \one{J_{ns}^3 \le 1/2}d\Pi_n^3(2\l_K s)\} \\
& - \sum_{n\ge 1} \int_0^t \one{n\le Y_{s-}} \{d\Pi^2_n((\mu - \m_K)s)
                           + \one{J_{ns}^4 \le 1/2}d\Pi_n^4(2\m_K s)\}; \\
V_t = V_0 + \sum_{n\ge 1} \int_0^t & \one{n\le V_{s-}}\,d\Pi^1_n((\l - \l_K)s) \\
&   - \sum_{n\ge 1} \int_0^t \one{n\le V_{s-}}\{d\Pi^2_n((\mu - \m_K)s)  + d\Pi_n^4(2\m_K s)\}
\end{align*}

That these representations yield the distributions of the corresponding BD processes follows because they define
Markov processes having the right jump rates.  These processes only have jumps of~$\pm 1$, so that, for two
of them to cross each other, there have to be times at which they have the same values.  However, if $U_t = Y_t$,
the next transition either leaves their values the same, or increases~$U$ by~$1$, leaving~$Y$ unchanged, or
reduces~$Y$ by~$1$, leaving~$U$ unchanged: so, if $U_0 \ge Y_0$, then $U_t \ge Y_t$ for all $t\ge0$.  The same
considerations yield $Y_t \ge V_t$ for all $t\ge0$, if $Y_0 \ge V_0$, and property~(f) follows,
assuming property~(a).

In order to define the process~$Z$, let
\[
     p^3(t) \Def (\l_K - (\l-\m)g_1(Z_{t-}/K))/(2\l_K);\quad p^4(t) \Def (\l_K + (\l-\m)g_2(Z_{t-}/K))/(2\m_K),
\]
noting that, if $0 \le t \le \t^{(K)} := \inf\{s > 0\colon\, Z_s \ge K^{\a+c+\h}\}$, as defined in property~(g),
then $0\le p^l(t) \le 1$, $l\in\{3,4\}$.  Then the process
\begin{multline*}
 Z_t = Z_0 + \sum_{n\ge 1} \int_0^t \one{n\le Z_{s-}}\{d\Pi^1_n((\l - \l_K)s)
                                           + \one{J_{ns}^3 \le p^3(s)}d\Pi_n^3(2\l_K s)\}\phantom{XXXXXXX} \\
         - \sum_{n\ge 1} \int_0^t \one{n\le Z_{s-}} \{d\Pi^2_n((\mu - \m_K)s) + \one{J_{ns}^4 \le p^4(s)}d\Pi_n^4(2\m_K s)\}
\end{multline*}
is Markovian and has the correct transition rates for $0 \le t \le \t^{(K)}$;  after that time, $Z$ can be
continued in any way that reproduces the correct distribution.  The argument used to show that $U_t \ge Y_t \ge V_t$
for all $t\ge0$ also shows that $U_t \ge Z_t \ge V_t$ for all $0\le t\le \t^{(K)}$, if $U_0 \ge Z_0 \ge V_0$,
and property~(g) follows, assuming property~(a).

\medskip

{\bf Acknowledgements}. ADB and FCK were supported in part by Australian Research Council Grants
Nos DP120102728 and DP150103588.
FCK thanks Haya Kaspi and Tom Kurtz for useful discussions.



\end{document}